\documentclass[12pt,bezier]{article}
\usepackage{times}
\usepackage{booktabs}
\usepackage{pifont}
\usepackage{floatrow}
\usepackage{caption}
\usepackage{mathrsfs}
\usepackage[fleqn]{amsmath}
\usepackage{amsfonts,amsthm,amssymb,mathrsfs,bbding}
\usepackage{txfonts}
\usepackage{graphics,multicol}
\usepackage{graphicx}
\usepackage{color}

\usepackage{caption}
\usepackage{cite}
\usepackage{latexsym,bm}
\usepackage{indentfirst}
\usepackage{color}
\usepackage[colorlinks=true,anchorcolor=blue,filecolor=blue,linkcolor=blue,urlcolor=blue,citecolor=blue]{hyperref}
\usepackage{extarrows}
\usepackage{cite}
\usepackage{latexsym,bm}
\usepackage{mathtools}
\evensidemargin=\oddsidemargin\topmargin=-1.5cm

\newtheorem{theorem}{Theorem}[section]
\newtheorem{problem}{Problem}[section]

\newtheorem{lemma}[theorem]{Lemma}
\newtheorem{corollary}{Corollary}[section]

\newtheorem{conj}[theorem]{Conjecture}

\theoremstyle{definition}

\addtocounter{section}{0}

\begin{document}
\title{Tur\'{a}n number of books in non-bipartite graphs}

\author{{\bf Lu Miao$^{a,b}$},~~{\bf Ruifang Liu$^{a}$}\thanks{Corresponding author. E-mail address: rfliu@zzu.edu.cn (R. Liu).}
~~{\bf Edwin R. van Dam$^{b}$}
\\
{\footnotesize $^a$ School of Mathematics and Statistics, Zhengzhou University, Zhengzhou, Henan 450001, China}\\
{\footnotesize $^b$ Department of Econometrics and O.R., Tilburg University, Tilburg, the Netherlands}}

\date{}
\maketitle
{\flushleft\large\bf Abstract}
Let $\mathrm{ex}(n, H)$ be the Tur\'{a}n number of $H$ for a given graph $H$. A graph is color-critical if it contains an edge whose removal reduces its chromatic number. Simonovits' chromatic critical edge theorem states that if $H$ is color-critical with $\chi(H)=k+1$, then there exists an $n_0(H)$ such that ex$(n, H)=e(T_{n,k})$ and the Tur\'{a}n graph $T_{n,k}$ is the only extremal graph provided $n\geq n_0(H).$ A book graph $B_{r+1}$ is a set of $r+1$ triangles with a common edge, where $r\geq0$ is an integer. Note that $B_{r+1}$ is a color-critical graph with $\chi(B_{r+1})=3$. Simonovits' theorem implies that $T_{n,2}$ is the only extremal graph for $B_{r+1}$-free graphs of sufficiently large order $n$. Furthermore, Edwards and independently Khad\v{z}iivanov and Nikiforov completely confirmed Erd\H{o}s' booksize conjecture and obtained that ex$(n, B_{r+1})=e(T_{n,2})$ for $n\geq n_0(B_{r+1})=6r$. Recently, Zhai and Lin [J. Graph Theory 102 (2023) 502-520] investigated the problem of booksize from a spectral perspective.

Note that the extremal graph $T_{n,2}$ is bipartite. Motivated by the above elegant results, we in this paper focus on the Tur\'{a}n problem of non-bipartite $B_{r+1}$-free graphs of order $n$. For $r = 0,$ Erd\H{o}s proved a nice result: If $G$ is a non-bipartite triangle-free graph on $n$ vertices, then $e(G)\leq\big\lfloor\frac{(n-1)^{2}}{4}\big\rfloor+1$. For general $r\geq1,$ we determine the exact value of Tur\'{a}n number of $B_{r+1}$ in non-bipartite graphs and characterize all extremal graphs provided $n$ is sufficiently large. An interesting phenomenon is that the Tur\'{a}n numbers and extremal graphs are completely different for $r=0$ and general $r\geq1.$

\begin{flushleft}
\end{flushleft}

\section{Introduction}\label{se1}
All graphs considered here are simple and undirected. We denote by $V(G)$ and $E(G)$ the vertex set and the edge set of $G$, respectively. Let $e(G)$ be the number of edges in $G$. For each vertex $u\in V(G)$, let $N(u)$ be the set of neighbours of $u$ in $V(G)$. The degree of $u$, denoted by $d(u)$, is the number of vertices in $N(u)$. For a subset $S\subseteq V(G)$, let $N_{S}(u)$ and $d_{S}(u)$ be the set and the number of neighbours of $u$ in $S$, respectively. For two disjoint subsets $A, B\subseteq V(G)$, we use $E(A,B)$ and $e(A,B)$ to denote the set and the number of edges with one endpoint in $A$ and the other in $B$, respectively. Let $G[S]$ be the subgraph of $G$ induced by $S$, and $|S|$ be the number of vertices in $S.$ Denote by $G\backslash S$ the subgraph of $G$ induced by $V(G)\backslash S$. For a nonempty subset $E'\subseteq E(G)$, let $G[E']$ be the edge-induced subgraph of $G$. We denote by $\chi(H)$ the chromatic number of $H.$

For a given graph $H$, if $G$ does not contain $H$ as a subgraph, then $G$ is {\it $H$-free}. The {\it Tur\'{a}n number} $\mathrm{ex}(n, H)$ is defined as the maximum number of edges in $H$-free graphs on $n$ vertices. An $H$-free graph of order $n$ with $\mathrm{ex}(n, H)$ edges is called an {\it extremal graph}, and we denote by $\mathrm{Ex}(n, H)$ the set of all extremal graphs on $n$ vertices. The Tur\'{a}n-type problem is a cornerstone of extremal graph theory, which aims to determine $\mathrm{ex}(n, H)$ and $\mathrm{Ex}(n, H)$ for various graphs $H$. Let $T_{n, k}$ be the {\it Tur\'{a}n graph} with $n$ vertices and $k$ parts. The study of the Tur\'{a}n-type problem can date back to Mantel's theorem \cite{Mantel1907} in 1907.
\begin{theorem}[\!\cite{Mantel1907}]\label{th1}
$\mathrm{ex}(n, C_3)=\big\lfloor\frac{n^{2}}{4}\big\rfloor$ and $\mathrm{Ex}(n, C_3)=\{T_{n, 2}\}$.
\end{theorem}

Denote by $\mathrm{ex}_{\chi(H)}(n, H)$ the maximum number of edges in non-$(\chi(H)-1)$-partite $H$-free graphs of order $n$. Let $\mathrm{Ex}_{\chi(H)}(n, H)$ be the set of graphs with $\mathrm{ex}_{\chi(H)}(n, H)$ edges. Note that the extremal graph $T_{n,2}$ in Theorem \ref{th1} is bipartite. Erd\H{o}s \cite{EP} further refined Mantel's theorem as follows.

\begin{theorem}[\!\cite{EP}]\label{th2}
$\mathrm{ex}_3(n, C_3)=\Big\lfloor\frac{(n-1)^{2}}{4}\Big\rfloor+1.$
\end{theorem}

Caccetta and Jia \cite{CJ2002} extended Theorem \ref{th2} and proved that $e(G)\leq\big\lfloor\frac{(n-2k+1)^{2}}{4}\big\rfloor+2k-1$ for a non-bipartite graph $G$ on $n$ vertices having no odd cycle of length at most $2k+1$. Furthermore, they also characterized all extremal graphs. Caccetta and Jia's result for $k=1$ indicates that
$$\mathrm{Ex}_{3}(n, C_3)=\mathscr{G}_{0}(C_3),$$ where the family of graphs $\mathscr{G}_{0}(C_3)$ is defined as follows.

\begin{figure}\label{f1}
\centering
\includegraphics[width=0.6\textwidth]{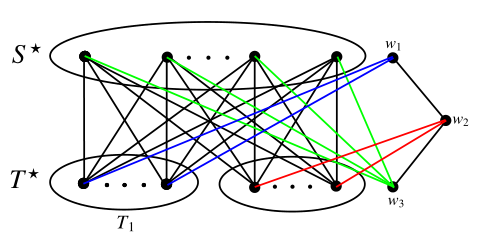}
\caption{Graphs in $\mathscr{G}_0(C_{3})$.}\label{f1}
\end{figure}

\vspace{3mm}
$\bigstar$ Let $S^{\star}$ and $T^{\star}$ be the two parts of $T_{n-3,2}$, and let $T_1$ be a non-empty proper subset of $T^{\star}$. Denote by $\mathscr{G}_0(C_{3})$ the family of graphs obtained from $T_{n-3,2}$ and $P_3=w_1w_2w_3$ by joining $w_1$ to all vertices in $T_1$, $w_2$ to all vertices in $T^{\star}\backslash T_1$ and $w_3$ to all vertices in $S^{\star}$ (see Fig. \ref{f1}).
\vspace{3mm}

The following result generalized Theorem \ref{th1} from $C_{3}$ to $C_{2k+1}$.

\begin{theorem}[\!\cite{B,B1971,FG,W1972}]\label{th0}
Let $k\geq2$ and $n\geq4k-2$ be integers. Then
$$\mathrm{ex}(n, C_{2k+1})=\Big\lfloor\frac{n^{2}}{4}\Big\rfloor~~\mbox{and}~~\mathrm{Ex}(n, C_{2k+1})=\{T_{n,2}\}.$$
\end{theorem}

Note that the extremal graph $T_{n,2}$ in Theorem \ref{th0} is bipartite. Ren, Wang, Wang and Yang's result (\cite{RWWY}, Theorem 1.3) implied the following theorem.
Denote by $T_{n-3,2}\bullet C_3$ the graph obtained from $T_{n-3,2}$ and $C_3$ by sharing a vertex.

\begin{theorem}[\!\cite{RWWY}]\label{th00}
Let $k\geq2$ and $n\geq318k$ be integers. Then
$$\mathrm{ex}_3(n, C_{2k+1})=\Big\lfloor\frac{(n-2)^{2}}{4}\Big\rfloor+3~~\mbox{and}~~\mathrm{Ex}_3(n, C_{2k+1})=\{T_{n-3,2}\bullet C_3\}.$$
\end{theorem}

In 1941, Tur\'{a}n \cite{Turan1941} generalized Mantel's theorem to the complete graph $K_{r+1}$.

\begin{theorem}[\!\cite{Turan1941}]\label{th3}
$\mathrm{ex}(n, K_{r+1})=e(T_{n, r})$ and $\mathrm{Ex}(n, K_{r+1})=\{T_{n, r}\}$.
\end{theorem}

Since the extremal graph $T_{n,r}$ in Theorem \ref{th3} is $r$-partite, Brouwer \cite{B1981} considered non-$r$-partite graphs and proposed a refinement of Tur\'{a}n's theorem.
\begin{theorem}[\!\cite{B1981}]\label{th4}
Let $n\geq2r+1$ be an integer. Then
$$\mathrm{ex}_{r+1}(n, K_{r+1})=e(T_{n,r})-\Big\lfloor\frac{n}{r}\Big\rfloor+1.$$
\end{theorem}
Theorem \ref{th4} was also independently studied in many references, see \cite{AFGS,KP,TU}. There are many extremal graphs attaining the maximum number of edges. Recently, Li and Peng \cite{LP2} proved a spectral version of Theorem \ref{th4}.

A graph is {\it color-critical} if it contains an edge whose removal reduces its chromatic number. Simonovits \cite{S1968,S1974} proved the following outstanding result, which is known as Simonovits' chromatic critical edge theorem.

\begin{theorem}[\!\cite{S1968,S1974}]\label{th5}
If $H$ is color-critical with $\chi(H)=k+1$, then there exists an $n_0(H)$ such that $\mathrm{ex}(n, H)=e(T_{n,k})$ and the Tur\'{a}n graph $T_{n,k}$ is the only extremal graph provided $n\geq n_0(H).$
\end{theorem}

Books, as a kind of color-critical graph, play an important role in the history of extremal graph theory. In 1962, Erd\H{o}s \cite{E} initiated the study of books in graphs. Since then books have attracted considerable attention both in extremal graph theory (see, e.g., \cite{EFG,EFR,KN}) and in Ramsey graph theory (see, e.g., \cite{NR2005,RS}).
A book graph $B_{r+1}$ is a set of $r+1$ triangles with a common edge, where $r\geq0$. It is easy to see that $B_{r+1}$ is a color-critical graph with $\chi(B_{r+1})=3$. Simonovits' chromatic critical edge theorem implies that $T_{n,2}$ is the only extremal graph for $B_{r+1}$-free graphs of sufficiently large order $n$. Furthermore, Edwards \cite{E1977} and independently Khad\v{z}iivanov and Nikiforov \cite{KN} completely confirmed Erd\H{o}s' booksize conjecture.
\begin{theorem}[\!\cite{E1977,KN}]\label{th6}
$\mathrm{ex}(n, B_{r+1})=e(T_{n,2})$ for $n\geq n_0(B_{r+1})=6r$.
\end{theorem}
\begin{figure}\label{f2}
\centering
\includegraphics[width=1.02\textwidth]{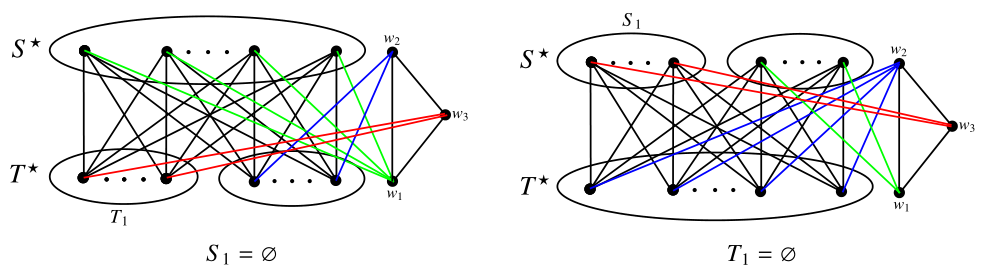}
\caption{$\mathscr{G}_1(B_{2})$ with $S_1=\varnothing$ or $T_1=\varnothing$.}\label{f2}
\end{figure}
\begin{figure}\label{f3}
\centering
\includegraphics[width=0.5\textwidth]{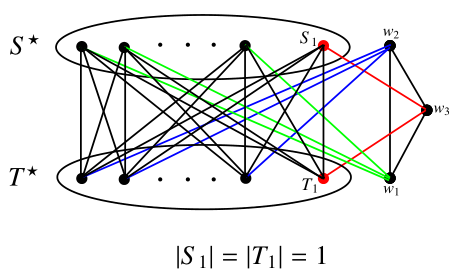}
\caption{$\mathscr{G}_1(B_{2})$ with $|S_1|=|T_1|=1$.}\label{f3}
\end{figure}
Note that the extremal graph $T_{n,2}$ is bipartite. Hence it is natural and interesting to consider the following problem.

\begin{problem}\label{p1}
What is the maximum number of edges in non-bipartite $B_{r+1}$-free graphs of order $n$? Moreover, characterize all extremal graphs.
\end{problem}

For $r=0$, $B_{r+1}$ is isomorphic to a triangle. Erd\H{o}s \cite{EP} determined the maximum number of edges in non-bipartite triangle-free graphs of order $n$, and the extremal graphs are completely characterized by Caccetta and Jia \cite{CJ2002}.
For $r\geq1$, we in this paper focus on Problem \ref{p1}, determine the exact value of Tur\'{a}n number of $B_{r+1}$ in non-bipartite graphs with $n$ vertices and characterize all extremal graphs for sufficiently large $n$.

Next we define the family of graphs $\mathscr{G}_1(B_{2})$ and the graph $K_{\lfloor\frac{n-1}{2}\rfloor,\lceil\frac{n-1}{2}\rceil}^{r, r}$.

\vspace{1mm}
$\bigstar$ Write the two parts of $T_{n-3,2}$ as $S^{\star}$ and $T^{\star}$. Let $S_1$ and $T_1$ be subsets of $S^{\star}$ and $T^{\star}$, respectively. Let $C_3=w_1w_2w_3w_1$ be a triangle. Define $\mathscr{G}_1(B_{2})$ as the family of graphs obtained from $T_{n-3,2}$ and $C_3$ by joining $w_3$ to all vertices in $S_1\cup T_1$, $w_1$ to all vertices in $S^{\star}\backslash S_1$ and $w_2$ to all vertices in $T^{\star}\backslash T_1$, where $|S_1||T_1|\leq1$. It follows from $|S_1||T_1|\leq1$ that $\mathscr{G}_1(B_{2})$ can be classified into three types of graphs: $S_1=\varnothing$, $T_1=\varnothing$ (see Fig. \ref{f2}) and $|S_1|=|T_1|=1$ (see Fig. \ref{f3}).

\vspace{1mm}
$\bigstar$ Denote by $K_{\lfloor\frac{n-1}{2}\rfloor,\lceil\frac{n-1}{2}\rceil}^{r, r}$ the graph obtained from the complete bipartite graph $K_{\lfloor\frac{n-1}{2}\rfloor,\lceil\frac{n-1}{2}\rceil}$ by adding a new vertex $v_{0}$ such that $v_{0}$ has $r$ neighbours in each part of $K_{\lfloor\frac{n-1}{2}\rfloor,\lceil\frac{n-1}{2}\rceil}$, where $r\geq1$ be an integer (see Fig. \ref{f4}).
\vspace{2mm}

Now we present our main result.

\begin{figure}\label{f4}
\centering
\includegraphics[width=0.6\textwidth]{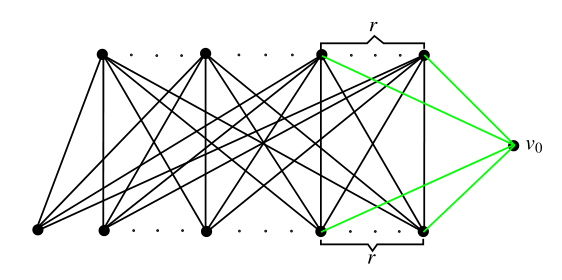}
\caption{The graph $K_{\lfloor\frac{n-1}{2}\rfloor,\lceil\frac{n-1}{2}\rceil}^{r, r}$.}\label{f4}
\end{figure}

\begin{theorem}\label{main}
Let $r\geq1$ be an integer and $n$ be sufficiently large. Then
\begin{eqnarray*}
\mathrm{ex}_3(n, B_{r+1})=\bigg\lfloor\frac{(n-1)^{2}}{4}\bigg\rfloor+2r.
\end{eqnarray*}
Moreover, $\mathrm{Ex}_3(n, B_{r+1})=\mathscr{G}_1(B_{2})$ if $r=1$ and $\mathrm{Ex}_3(n, B_{r+1})=\big\{K_{\lfloor\frac{n-1}{2}\rfloor,\lceil\frac{n-1}{2}\rceil}^{r, r}\big\}$ if $r\geq2$.
\end{theorem}

From Theorems \ref{th2} and \ref{main}, one can observe that the Tur\'{a}n numbers and extremal graphs are very different for $r=0$ and general $r\geq1$.

Let $A(G)$ be the adjacency matrix of $G$. The largest eigenvalue of $A(G)$, denoted by $\rho(G)$, is called the spectral radius of $G$. Let $\mathrm{Spex}(n, H)$ be the set of $n$-vertex $H$-free graphs with maximum spectral radius. In 2022, Cioab\u{a}, Desai and Tait \cite{CDT} proposed the following conjecture.

\begin{conj}[\!\cite{CDT}]\label{con}
If $H$ is a graph such that the graphs in $\mathrm{Ex}(n,H)$ are Tur\'{a}n graph plus $O(1)$ edges, then $\mathrm{Spex}(n,H)\subseteq\mathrm{Ex}(n,H)$ for sufficiently large $n$.
\end{conj}

Until now, there has been many positive evidences for Conjecture \ref{con} such as $W_5$ \cite{CDT}, complete graphs \cite{Nikiforov2007,G1996}, friendship graphs \cite{CFTZ,ZLX}, intersecting cliques \cite{DKL,MLZ}, intersecting odd cycles \cite{LP,MLZ}, $k$-edge-disjoint-triangles \cite{LZZ2022} and color-critical graphs \cite{N2009}. Recently, Wang, Kang and Xue \cite{WKX} completely solved Conjecture \ref{con} by proving a stronger theorem.
\begin{theorem}[\!\cite{WKX}]\label{th9}
Let $r\geq2$ be an integer, and $H$ be a graph with $\mathrm{ex}(n, H)=e(T_{n, r})+O(1).$ Then $\mathrm{Spex}(n,H)\subseteq\mathrm{Ex}(n,H)$ for sufficiently large $n$.
\end{theorem}

Initially, Zhai and Lin \cite{ZL2023} investigated the spectral Tur\'{a}n problem of $B_{r+1}$ in general graphs on $n$ vertices.
\begin{theorem}[\!\cite{ZL2023}]\label{th10}
Let $n$ and $r$ be integers with $n\geq\frac{13}{2}r$. If $G$ is a $B_{r+1}$-free graph of order $n$, then
$\rho(G)\leq\rho(T_{n,2})$ with equality if and only if $G\cong T_{n,2}$.
\end{theorem}

Note that the above extremal graph $T_{n,2}$ is bipartite. For $r=0$, Lin, Ning and Wu \cite{LNW} characterized the spectral extremal graph in non-bipartite $C_3$-free graphs. For $r\geq1$, Liu and Miao \cite{LM} determined the maximum spectral radius in non-bipartite $B_{r+1}$-free graphs of order $n$ and characterized the spectral extremal graph. Let $\mathrm{Spex}_{3}(n, B_{r+1})$ be the set of spectral extremal graphs.
\begin{theorem}[\!\cite{LM}]\label{th11}
Let $r\geq1$ and $n\geq8(r^{2}+r+2)$ be integers. Then
\begin{eqnarray*}
\mathrm{Spex}_{3}(n, B_{r+1})=\big\{K_{\lfloor\frac{n-1}{2}\rfloor,\lceil\frac{n-1}{2}\rceil}^{r, r}\big\}.
\end{eqnarray*}
\end{theorem}

Note that $K_{\lfloor\frac{n-1}{2}\rfloor,\lceil\frac{n-1}{2}\rceil}^{1, 1}\in\mathscr{G}_1(B_{2})$ by setting $S_1=T_1=\varnothing$. By Theorems \ref{main} and \ref{th11}, an immediate corollary follows directly.

\begin{corollary}
Let $r\geq1$ and $n$ be sufficiently large. If $G$ is a non-bipartite $B_{r+1}$-free graph of order $n$, then
$$\mathrm{Spex}_{3}(n, B_{r+1})\subseteq\mathrm{Ex}_{3}(n, B_{r+1}).$$
\end{corollary}

\section{Proof of Theorem \ref{main}}\label{se2}
In this section, we first present some important results, which are essential to the proof of our main theorem.
The following is the classical Erd\H{o}s-Simonovits stability theorem.

\begin{lemma}[\!\cite{E1966,E1968,S1968}]\label{lem3}
Let $H$ be a graph with $\chi(H)=r+1\geq3$. For every $\epsilon>0$, there exist a constant $\delta>0$ and an integer $n_0$ such that if $G$ is an $H$-free graph on $n\geq n_0$ vertices with $e(G)\geq(1-\frac{1}{r}-\delta)\frac{n^{2}}{2}$, then $G$ can be obtained from $T_{n,r}$ by adding and deleting at most $\epsilon n^{2}$ edges.
\end{lemma}

The following lemma can be proved by induction or double counting.

\begin{lemma}[\!\cite{CFTZ}]\label{lem4}
If $S_1,S_2,\ldots,S_k$ are $k$ finite sets, then $|S_1\cap S_2\cap\cdots\cap S_k|=\sum_{i=1}^{k}|S_i|-(k-1)|\bigcup_{i=1}^{k}S_i|$.
\end{lemma}

Let $G$ be an arbitrary graph in $\mathrm{Ex}_{3}(n, B_{r+1})$. Next our goal is to obtain the asymptotic structure of $G$. First, we can claim that $K_{\lfloor\frac{n-1}{2}\rfloor,\lceil\frac{n-1}{2}\rceil}^{r, r}$ is $B_{r+1}$-free. In fact, assume that $K_{\lfloor\frac{n-1}{2}\rfloor,\lceil\frac{n-1}{2}\rceil}^{r, r}$ contains a copy of $B_{r+1}$, say $H$. Note that $K_{\lfloor\frac{n-1}{2}\rfloor,\lceil\frac{n-1}{2}\rceil}^{r, r}\backslash\{v_0\}$ is a bipartite graph. Then $v_0\in V(H).$ If $d_{H}(v_0)=2$, then $H\backslash\{v_0\}\cong B_{r}$ is a subgraph of $K_{\lfloor\frac{n-1}{2}\rfloor,\lceil\frac{n-1}{2}\rceil}^{r, r}\backslash\{v_0\}$, a contradiction. Hence $d_{H}(v_0)=r+2$. Since $H[N_{H}(v_0)]$ is isomorphic to $K_{1,r+1}$, $K_{1,r+1}$ is a subgraph of $G[N_{G}(v_0)]$, which contradicts that $G[N_{G}(v_0)]\cong K_{r,r}$ (see Fig. \ref{f4}). Hence $K_{\lfloor\frac{n-1}{2}\rfloor,\lceil\frac{n-1}{2}\rceil}^{r, r}$ is $B_{r+1}$-free. Note that $K_{\lfloor\frac{n-1}{2}\rfloor,\lceil\frac{n-1}{2}\rceil}^{r, r}$ is also non-bipartite. Then we have
\begin{eqnarray}\label{eq1}
e(G)\geq e\Big(K_{\lfloor\frac{n-1}{2}\rfloor,\lceil\frac{n-1}{2}\rceil}^{r,r}\Big)=\bigg\lfloor\frac{(n-1)^{2}}{4}\bigg\rfloor+2r\geq\frac{n^{2}}{4}-\frac{n}{2}+2r
=\bigg(\frac{1}{2}-\Big(\frac{1}{n}-\frac{4r}{n^{2}}\Big)\bigg)\frac{n^{2}}{2}.
\end{eqnarray}
Let $\epsilon$ be a sufficiently small constant satisfying
\begin{eqnarray}\label{eq1.0}
\max\Big\{60r\sqrt{\epsilon},~90\sqrt{\epsilon}\Big\}<1.
\end{eqnarray}

\begin{lemma}\label{cla1}
For every $\epsilon>0$, there exists an integer $n_0$ such that if $n\geq n_0$, then
\begin{eqnarray*}
e(G)\geq\Big(\frac{1}{4}-\epsilon\Big)n^{2}.
\end{eqnarray*}
Moreover, $G$ has a partition $V(G)=S\cup T$ which gives a maximum cut with
\begin{eqnarray*}
e(S)+e(T)\leq\epsilon n^{2}
\end{eqnarray*}
and
\begin{eqnarray*}
\Big(\frac{1}{2}-\frac{3}{2}\sqrt{\epsilon}\Big)n\leq |S|,|T|\leq\Big(\frac{1}{2}+\frac{3}{2}\sqrt{\epsilon}\Big)n.
\end{eqnarray*}
\end{lemma}
\begin{proof}
Note that $\chi(B_{r+1})=3$. Combining (\ref{eq1}) and Lemma \ref{lem3}, for every $\epsilon>0$, we can take a large enough $n$ such that $G$ can be obtained from $T_{n,2}$ by adding and deleting at most $\epsilon n^{2}$ edges. Then $$e(G)\geq\Big(\frac{1}{4}-\epsilon\Big)n^{2},$$ and there exists a partition of $V(G)=S_0\cup T_0$ with $\lfloor\frac{n}{2}\rfloor\leq|S_0|\leq|T_0|\leq\lceil\frac{n}{2}\rceil$ such that $e(S_0)+e(T_0)\leq\epsilon n^{2}$. We choose a partition of $V(G)=S\cup T$ such that $e(S, T)$ is a maximum cut. Then we have $$e(S)+e(T)\leq e(S_0)+e(T_0)\leq\epsilon n^{2}.$$
Now we assume that $|S|=\frac{n}{2}-\beta$ and $|T|=\frac{n}{2}+\beta$ for some $\beta$. Then
\begin{eqnarray*}
e(G)=e(S, T)+e(S)+e(T)\leq|S||T|+e(S)+e(T)\leq\frac{n^{2}}{4}-\beta^{2}+\epsilon n^{2}.
\end{eqnarray*}
Recall that $e(G)\geq\Big(\frac{1}{4}-\epsilon\Big)n^{2}$. Then $\beta^{2}\leq2\epsilon n^{2}$. Hence $\beta<\frac{3}{2}\sqrt{\epsilon}n$, which implies that
\begin{eqnarray*}
\Big(\frac{1}{2}-\frac{3}{2}\sqrt{\epsilon}\Big)n\leq |S|,|T|\leq\Big(\frac{1}{2}+\frac{3}{2}\sqrt{\epsilon}\Big)n.
\end{eqnarray*}
Hence the lemma holds.
\end{proof}

In the following, we define two vertex subsets $L$ and $W$ of $G$ and estimate upper bounds of the number of vertices in $L$ and $W$, respectively.

\begin{lemma}\label{cla2}
Let $L=\big\{v\in V(G)\mid d_{G}(v)\leq\big(\frac{1}{2}-4\sqrt{\epsilon}\big)n\big\}$. Then $|L|\leq\sqrt{\epsilon}n$.
\end{lemma}
\begin{proof}
Suppose to the contrary that $|L|>\sqrt{\epsilon}n$. Then there exists a subset $L'$ of $L$ such that $|L'|=\lfloor\sqrt{\epsilon}n\rfloor$.
By Lemma \ref{cla1} and the definition of $L$, we have
\begin{eqnarray*}
e(G\backslash L')\geq e(G)-\sum_{v\in L'}d_{G}(v)&\geq&\Big(\frac{1}{4}-\epsilon\Big)n^{2}-\sqrt{\epsilon}n^{2}\Big(\frac{1}{2}-4\sqrt{\epsilon}\Big)\\
&=&\frac{1}{4}(1-2\sqrt{\epsilon}+12\epsilon)n^{2}.
\end{eqnarray*}
Note that $G\backslash L'$ is $B_{r+1}$-free and $B_{r+1}$ is a color-critical graph with $\chi(B_{r+1})=3$. By Theorem \ref{th5}, for sufficiently large $n$, we have
\begin{eqnarray*}
e(G\backslash L')\leq e(T_{|V(G\backslash L')|,2})&\leq&\frac{(n-\lfloor\sqrt{\epsilon}n\rfloor)^{2}}{4}<\frac{(n-\sqrt{\epsilon}n+1)^{2}}{4}\\
&<&\frac{1}{4}(1-2\sqrt{\epsilon}+3\epsilon)n^{2},
\end{eqnarray*}
a contradiction. Hence $|L|\leq\sqrt{\epsilon}n$.
\end{proof}

\begin{lemma}\label{cla3}
Define $W_1=\big\{v\in S\mid d_{S}(v)\geq\frac{7}{2}\sqrt{\epsilon}n\big\}$ and $W_2=\big\{v\in T\mid d_{T}(v)\geq\frac{7}{2}\sqrt{\epsilon}n\big\}$. Let $W=W_1\cup W_2$. Then $|W|\leq\frac{4}{7}\sqrt{\epsilon}n$.
\end{lemma}
\begin{proof}
By the definition of $W_1$ and $W_2$, we obtain that $2e(S)=\sum_{v\in S}d_{S}(v)\geq\sum_{v\in W_1}d_{S}(v)\geq\frac{7}{2}|W_1|\sqrt{\epsilon}n$ and
$2e(T)=\sum_{v\in T}d_{T}(v)\geq\sum_{v\in W_2}d_{T}(v)\geq\frac{7}{2}|W_2|\sqrt{\epsilon}n$. Hence
\begin{eqnarray}\label{eq2}
e(S)+e(T)\geq\frac{7}{4}(|W_1|+|W_2|)\sqrt{\epsilon}n=\frac{7}{4}|W|\sqrt{\epsilon}n.
\end{eqnarray}
By Lemma \ref{cla1}, $e(S)+e(T)\leq\epsilon n^{2}$. Combining (\ref{eq2}), we have $|W|\leq\frac{4}{7}\sqrt{\epsilon}n$.
\end{proof}

To determine the relationship between $L$ and $W$, we first prove the following lemma. Let $\widetilde{S}=S\backslash(L\cup W)$ and $\widetilde{T}=T\backslash(L\cup W)$.

\begin{lemma}\label{cla4}
For every vertex $u\in V(G)$ with $d_{G}(u)\geq20\sqrt{\epsilon}n$, either $N_{\widetilde{S}}(u)=\varnothing$ or $N_{\widetilde{T}}(u)=\varnothing$.
\end{lemma}
\begin{proof}
Let $u$ be an arbitrary vertex in $V(G)$ with $d_{G}(u)\geq20\sqrt{\epsilon}n$. Then either $d_{S}(u)\geq10\sqrt{\epsilon}n$ or $d_{T}(u)\geq10\sqrt{\epsilon}n$. Without loss of generality, we assume that
\begin{eqnarray}\label{eq3}
d_{S}(u)\geq10\sqrt{\epsilon}n.
\end{eqnarray}
Next it suffices to prove that $N_{\widetilde{T}}(u)=\varnothing$. Suppose to the contrary that there exists a vertex $u_0\in N_{\widetilde{T}}(u)$. Recall that $\widetilde{T}=T\backslash(L\cup W)$. By the definitions of $L$ and $W_2$, $d_{G}(u_{0})>\big(\frac{1}{2}-4\sqrt{\epsilon}\big)n$ and $d_{T}(u_{0})<\frac{7}{2}\sqrt{\epsilon}n$. Then we obtain that
\begin{eqnarray}\label{eq4}
d_{S}(u_{0})=d_{G}(u_{0})-d_{T}(u_{0})>\Big(\frac{1}{2}-4\sqrt{\epsilon}\Big)n-\frac{7}{2}\sqrt{\epsilon}n>\Big(\frac{1}{2}-8\sqrt{\epsilon}\Big)n.
\end{eqnarray}
Combining (\ref{eq3}), (\ref{eq4}) and Lemma \ref{cla1}, for sufficiently large $n$, we have
\begin{eqnarray}\label{eq4.0}
|N_{S}(u)\cap N_{S}(u_{0})|&\geq&d_{S}(u)+d_{S}(u_{0})-|S|\nonumber\\
&>&10\sqrt{\epsilon}n+\Big(\frac{1}{2}-8\sqrt{\epsilon}\Big)n-\Big(\frac{1}{2}+\frac{3}{2}\sqrt{\epsilon}\Big)n\nonumber\\
&=&\frac{1}{2}\sqrt{\epsilon}n\geq r+1.
\end{eqnarray}
Recall that $u_0\in N_{\widetilde{T}}(u)$. Then $uu_{0}$ is an edge of $G$. Combining (\ref{eq4.0}), we can find that $G$ contains a book $B_{r+1}$, a contradiction.
\end{proof}

\begin{lemma}\label{cla5}
$W\subseteq L$.
\end{lemma}
\begin{proof}
Suppose that $W\backslash L\neq\varnothing$. Recall that $W=W_1\cup W_2$. Without loss of generality, we assume that $W_1\backslash L\neq\varnothing$. Then there exists a vertex $v_0\in W_1\backslash L$. By the definitions of $L$ and $W_1$, we have $d_{G}(v_0)>\big(\frac{1}{2}-4\sqrt{\epsilon}\big)n$ and $d_{S}(v_0)\geq\frac{7}{2}\sqrt{\epsilon}n$.
It follows from Lemmas \ref{cla2} and \ref{cla3} that
\begin{eqnarray}\label{eq5}
|N_{\widetilde{S}}(v_0)|\geq d_{S}(v_0)-|L|-|W|\geq\frac{7}{2}\sqrt{\epsilon}n-\sqrt{\epsilon}n-\frac{4}{7}\sqrt{\epsilon}n>0.
\end{eqnarray}
Note that $e(S, T)$ is a maximum cut and $v_0\in W_1\backslash L\subseteq S$. Then $d_{S}(v_0)\leq\frac{1}{2}d_{G}(v_0)$, and hence $d_{T}(v_0)=d_{G}(v_0)-d_{S}(v_0)\geq\frac{1}{2}d_{G}(v_0)>\big(\frac{1}{4}-2\sqrt{\epsilon}\big)n$.
By Lemmas \ref{cla2} and \ref{cla3},
\begin{eqnarray}\label{eq6}
|N_{\widetilde{T}}(v_0)|\geq d_{T}(v_0)-|L|-|W|\geq\Big(\frac{1}{4}-2\sqrt{\epsilon}\Big)n-\sqrt{\epsilon}n-\frac{4}{7}\sqrt{\epsilon}n>0.
\end{eqnarray}
Moreover, by (\ref{eq1.0}), we have
\begin{eqnarray}\label{eq7}
d_{G}(v_0)>\big(\frac{1}{2}-4\sqrt{\epsilon}\big)n>20\sqrt{\epsilon}n.
\end{eqnarray}
By (\ref{eq5}), (\ref{eq6}) and (\ref{eq7}), there exists a vertex $v_0\in V(G)$ with $d_{G}(v_0)>20\sqrt{\epsilon}n$ satisfying $N_{\widetilde{S}}(v_0)\neq\varnothing$ and $N_{\widetilde{T}}(v_0)\neq\varnothing$, which contradicts Lemma \ref{cla4}.
\end{proof}

Recall that $G$ is non-bipartite. Let $C$ be a shortest odd cycle of $G$. Now we determine the length of $C$.

\begin{lemma}\label{cla6}
$|V(C)|=3$.
\end{lemma}
\begin{proof}
If $|V(C)|\geq5$, then $G$ is a non-bipartite triangle-free graph. By Theorem \ref{th2}, $e(G)\leq\big\lfloor\frac{(n-1)^{2}}{4}\big\rfloor+1$. Combining (\ref{eq1}), we have $e(G)\geq\big\lfloor\frac{(n-1)^{2}}{4}\big\rfloor+2r$, a contradiction.
\end{proof}

In order to further investigate the relationship between $L$ and $V(C)$, we need to prove the following result.

\begin{lemma}\label{cla7}
Let $u_i\in\widetilde{T}$, where $2\leq i\leq r+2$. Then $u_1,\ldots,u_i$ have at least $\frac{1}{2}\sqrt{\epsilon}n$ common neighbours in $S$.
\end{lemma}
\begin{proof}
For $2\leq i\leq r+2$, it follows from (\ref{eq4}) that $d_{S}(u_{i})>\Big(\frac{1}{2}-8\sqrt{\epsilon}\Big)n$. Combining Lemmas \ref{lem4}, \ref{cla1} and (\ref{eq1.0}), we have
\begin{eqnarray*}
\Big|\bigcap_{k=1}^{i}N_{S}(u_{i})\Big|&\geq&\sum_{k=1}^{i}d_{S}(u_{i})-(i-1)|S|\\
&>&i\Big(\frac{1}{2}-8\sqrt{\epsilon}\Big)n-(i-1)\Big(\frac{1}{2}+\frac{3}{2}\sqrt{\epsilon}\Big)n\\
&\geq&\frac{1}{2}n-\frac{19}{2}(r+2)\sqrt{\epsilon}n+\frac{3}{2}\sqrt{\epsilon}n\\
&>&\frac{1}{2}\sqrt{\epsilon}n,
\end{eqnarray*}
which implies that $u_1,\ldots,u_i$ have at least $\frac{1}{2}\sqrt{\epsilon}n$ common neighbours in $S$.
\end{proof}

Now we can determine the relationship between $L$ and $V(C).$

\begin{lemma}\label{cla8}
$L\subseteq V(C)$.
\end{lemma}
\begin{proof}
Suppose to the contrary that $L\backslash V(C)\neq\varnothing$. Then there exists a vertex $v_0\in L\backslash V(C)$. By the definition of $L$, we know that
\begin{eqnarray}\label{eq8}
d_{G}(v_0)\leq\Big(\frac{1}{2}-4\sqrt{\epsilon}\Big)n.
\end{eqnarray}
Define $G'=G-\{v_0u\mid u\in N_{G}(v_0)\}+\big\{v_0u\mid u\in\widetilde{T}\backslash\{v_0\}\big\}$. Since $v_0\notin V(C)$, $G'$ still contains $C$ as a subgraph, and so $G'$ is non-bipartite. Now we claim that $G'$ is also $B_{r+1}$-free. Otherwise, $G'$ contains a copy of $B_{r+1}$, say $H$. Then $v_0\in V(H)$. Let $u_1,\ldots,u_i$ be the neighbours of $v_0$ in $H$, where $i=2$ or $r+2$. By the definition of $G',$ $u_1,\ldots,u_i\in\widetilde{T}$. By Lemma \ref{cla7}, $u_1,\ldots,u_i$ have at least $\frac{1}{2}\sqrt{\epsilon}n$ common neighbours in $S$. Hence we can select a vertex $u_0\in S\backslash V(H)$ such that $u_0$ is adjacent to $u_1,\ldots,u_i$. It is easy to see that $G[(V(H)\backslash\{v_0\})\cup\{u_0\}]$ contains a copy of $B_{r+1}$, a contradiction. Therefore, $G'$ is a non-bipartite $B_{r+1}$-free graph.

By Lemma \ref{cla1}, $|T|\geq\big(\frac{1}{2}-\frac{3}{2}\sqrt{\epsilon}\big)n$. Combining Lemmas \ref{cla2} and \ref{cla3}, we have
\begin{eqnarray}\label{eq9}
|\widetilde{T}\backslash\{v_0\}|\geq|T|-|L|-|W|-1\geq\big(\frac{1}{2}-\frac{3}{2}\sqrt{\epsilon}\big)n-\sqrt{\epsilon}n-\frac{4}{7}\sqrt{\epsilon}n-1>\Big(\frac{1}{2}-\frac{7}{2}\sqrt{\epsilon}\Big)n.
\end{eqnarray}
By (\ref{eq8}) and (\ref{eq9}), we obtain that $$e(G')=e(G)+|\widetilde{T}\backslash\{v_0\}|-d_{G}(v_0)>e(G)+\frac{1}{2}\sqrt{\epsilon}n>e(G),$$ which contradicts the maximality of $e(G).$ Hence $L\subseteq V(C)$.
\end{proof}

Let $G^{\star}=G\backslash V(C)$, $S^{\star}=S\backslash V(C)$ and $T^{\star}=T\backslash V(C)$. Combining Lemmas \ref{cla5}, \ref{cla8} and the definitions of $\widetilde{S}$ and $\widetilde{T}$, we know that $S^{\star}\subseteq\widetilde{S}$ and $T^{\star}\subseteq\widetilde{T}$. Recall that $|V(C)|=3$. Let $C=w_1w_2w_3w_1$. Without loss of generality, we assume that $d_{G^{\star}}(w_1)\geq d_{G^{\star}}(w_2)\geq d_{G^{\star}}(w_3)$. Now we estimate lower bounds of $d_{G^{\star}}(w_1)$ and $d_{G^{\star}}(w_2)$.

\begin{lemma}\label{cla9}
$d_{G^{\star}}(w_1)>\frac{n}{3}-2$ and $d_{G^{\star}}(w_2)\geq\frac{n}{4}-1$.
\end{lemma}
\begin{proof}
First, we consider the lower bound of $d_{G^{\star}}(w_1)$. Note that $d_{G}(w_i)=d_{G^{\star}}(w_i)+2$ for $i=1,2,3$. Then we have
\begin{eqnarray}\label{eq10}
e(G)=e(G^{\star})+\sum_{i=1}^{3}d_{G}(w_i)-3=e(G^{\star})+\sum_{i=1}^{3}d_{G^{\star}}(w_i)+3.
\end{eqnarray}
Since $G^{\star}$ is $B_{r+1}$-free and $|V(G^{\star})|=n-3$, we have $e(G^{\star})\leq\big\lfloor\frac{(n-3)^{2}}{4}\big\rfloor$ by Theorem \ref{th5}. By (\ref{eq1}), $e(G)\geq\lfloor\frac{(n-1)^{2}}{4}\big\rfloor+2r$. Combining (\ref{eq10}) and $r\geq1$, we obtain that
\begin{eqnarray*}
3d_{G^{\star}}(w_1)&\geq&\sum_{i=1}^{3}d_{G^{\star}}(w_i)=e(G)-e(G^{\star})-3\\
&\geq&\frac{n^{2}-2n}{4}+2r-\frac{(n-3)^{2}}{4}-3\\
&>&n-6.
\end{eqnarray*}
Hence $d_{G^{\star}}(w_1)>\frac{n}{3}-2$.

Next we shall estimate the lower bound of $d_{G^{\star}}(w_2)$. Define $G'=G-\{w_2,w_3\}$. Then
\begin{eqnarray}\label{eq10.0}
e(G)=e(G')+d_{G}(w_2)+d_{G}(w_3)-1=e(G')+d_{G^{\star}}(w_2)+d_{G^{\star}}(w_3)+3.
\end{eqnarray}
Note that $G'$ is $B_{r+1}$-free and $|V(G')|=n-2$. Then by Theorem \ref{th5}, $e(G')\leq\big\lfloor\frac{(n-2)^{2}}{4}\big\rfloor$. Combining (\ref{eq10.0}), $e(G)\geq\lfloor\frac{(n-1)^{2}}{4}\big\rfloor+2r$ and $r\geq1$, we have
\begin{eqnarray*}
2d_{G^{\star}}(w_2)&\geq&d_{G^{\star}}(w_2)+d_{G^{\star}}(w_3)=e(G)-e(G')-3\\
&\geq&\frac{n^{2}-2n}{4}+2r-\frac{(n-2)^{2}}{4}-3\\
&\geq&\frac{n}{2}-2,
\end{eqnarray*}
which means that $d_{G^{\star}}(w_2)\geq\frac{n}{4}-1$.
\end{proof}

By Lemma \ref{cla9} and (\ref{eq1.0}), $d_{G^{\star}}(w_1)>\frac{n}{3}-2>20\sqrt{\epsilon}n$ and $d_{G^{\star}}(w_2)\geq\frac{n}{4}-1>20\sqrt{\epsilon}n.$
It follows from Lemma \ref{cla4} that either $N_{\widetilde{S}}(w_{i})=\varnothing$ or $N_{\widetilde{T}}(w_{i})=\varnothing$ for $i=1,2$. Without loss of generality, assume that $N_{\widetilde{T}}(w_{1})=\varnothing$, that is, $N_{G^{\star}}(w_1)\subseteq S^{\star}.$ Now we can further prove that $N_{\widetilde{S}}(w_{2})=\varnothing$, which implies that $N_{G^{\star}}(w_2)\subseteq T^{\star}.$

\begin{lemma}\label{cla10}
$N_{\widetilde{S}}(w_{2})=\varnothing$.
\end{lemma}
\begin{proof}
Suppose to the contrary that $N_{\widetilde{S}}(w_{2})\neq\varnothing.$ Then $N_{\widetilde{T}}(w_{2})=\varnothing$, which implies that $N_{G^{\star}}(w_2)\subseteq S^{\star}\subseteq\widetilde{S}\subseteq S.$ Hence $|N_{G^{\star}}(w_2)|=|N_{S^{\star}}(w_2)|$. Recall that $|N_{G^{\star}}(w_1)|=|N_{S^{\star}}(w_1)|$. Combining (\ref{eq1.0}), Lemmas \ref{cla1} and \ref{cla9}, we obtain that
\begin{eqnarray*}
|N_{S^{\star}}(w_1)\cap N_{S^{\star}}(w_2)|&\geq&|N_{S^{\star}}(w_1)|+|N_{S^{\star}}(w_2)|-|S|\\
&>&\frac{n}{3}+\frac{n}{4}-3-\Big(\frac{1}{2}+\frac{3}{2}\sqrt{\epsilon}\Big)n\\
&\geq&\frac{n}{12}-2\sqrt{\epsilon}n>\frac{1}{2}\sqrt{\epsilon}n\\
&\geq&r+1,
\end{eqnarray*}
which implies that $G$ contains a $B_{r+1}$, a contradiction. Hence $N_{\widetilde{S}}(w_{2})=\varnothing$.
\end{proof}

Now we know that $N_{G^{\star}}(w_1)\subseteq S^{\star}$ and $N_{G^{\star}}(w_2)\subseteq T^{\star}$. Next we shall consider $N_{G^{\star}}(w_3)$. Note that $d_{G^{\star}}(w_3)=d_{S^{\star}}(w_3)+d_{T^{\star}}(w_3)$. Then we have the following result.

\begin{lemma}\label{cla11}
$d_{S^{\star}}(w_1)+d_{S^{\star}}(w_3)\leq|S^{\star}|+r-1$ and $d_{T^{\star}}(w_2)+d_{T^{\star}}(w_3)\leq|T^{\star}|+r-1$.
\end{lemma}
\begin{proof}
Assume that $d_{S^{\star}}(w_1)+d_{S^{\star}}(w_3)\geq|S^{\star}|+r$. Then
\begin{eqnarray*}
|N_{S^{\star}}(w_1)\cap N_{S^{\star}}(w_3)|&\geq&|N_{S^{\star}}(w_1)|+|N_{S^{\star}}(w_3)|-|S^{\star}|\geq r.
\end{eqnarray*}
Note that $w_1w_3\in E(G)$ and $w_2\in N_{G}(w_1)\cap N_{G}(w_3)$. Then $|N_{G}(w_1)\cap N_{G}(w_3)|\geq r+1$, which implies that $G$ contains a $B_{r+1}$, a contradiction.
Hence $d_{S^{\star}}(w_1)+d_{S^{\star}}(w_3)\leq|S^{\star}|+r-1$. Similarly, we can prove that $d_{T^{\star}}(w_2)+d_{T^{\star}}(w_3)\leq|T^{\star}|+r-1$.
\end{proof}

Now we are in a position to present the proof of Theorem \ref{main}.

\medskip
\noindent  \textbf{Proof of Theorem \ref{main}.}
Since $N_{G^{\star}}(w_1)\subseteq S^{\star}$ and $N_{G^{\star}}(w_2)\subseteq T^{\star}$, we have $d_{G^{\star}}(w_1)=d_{S^{\star}}(w_1)$ and $d_{G^{\star}}(w_2)=d_{T^{\star}}(w_2)$. Note that $|S^{\star}|+|T^{\star}|=n-3$. By Lemma \ref{cla11},
\begin{eqnarray}\label{eq11}
e\big(V(C),V(G^{\star})\big)&=&\sum_{i=1}^{3}d_{G^{\star}}(w_i)=d_{S^{\star}}(w_1)+d_{T^{\star}}(w_2)+d_{S^{\star}}(w_3)+d_{T^{\star}}(w_3)\nonumber\\
&\leq&|S^{\star}|+|T^{\star}|+2r-2=n+2r-5.
\end{eqnarray}
Note that $G^{\star}$ is $B_{r+1}$-free. By Theorem \ref{th5}, $e(G^{\star})\leq\big\lfloor\frac{(n-3)^{2}}{4}\big\rfloor$. Combining (\ref{eq11}), we have
\begin{eqnarray}\label{eq12}
e(G)&=&e(C)+e(V(C),V(G^{\star}))+e(G^{\star})\nonumber\\
&\leq&3+n+2r-5+\Big\lfloor\frac{(n-3)^{2}}{4}\Big\rfloor\nonumber\\
&=&\Big\lfloor\frac{(n-1)^{2}}{4}\Big\rfloor+2r.
\end{eqnarray}
By (\ref{eq1}), $e(G)\geq\big\lfloor\frac{(n-1)^{2}}{4}\big\rfloor+2r$. Combining (\ref{eq12}), we obtain that $e(G)=\big\lfloor\frac{(n-1)^{2}}{4}\big\rfloor+2r$. Hence the above all inequalities must be equalities. This implies that $G^{\star}\cong T_{n-3,2}$,
\begin{eqnarray}\label{eq13}
d_{S^{\star}}(w_1)+d_{S^{\star}}(w_3)=|S^{\star}|+r-1~~\mbox{and}~~d_{T^{\star}}(w_2)+d_{T^{\star}}(w_3)=|T^{\star}|+r-1.
\end{eqnarray}
Next we divide the proof into the following two cases according to different values of $r$.

\vspace{2mm}
\noindent{\bf Case 1.} $r=1$.
\vspace{2mm}

By (\ref{eq13}), we have $d_{S^{\star}}(w_1)+d_{S^{\star}}(w_3)=|S^{\star}|$ and $d_{T^{\star}}(w_2)+d_{T^{\star}}(w_3)=|T^{\star}|$. Note that $G$ is $B_2$-free, $w_1w_3\in E(G)$ and $w_2\in N_{G}(w_1)\cap N_{G}(w_3)$. Then $$N_{S^{\star}}(w_1)\cap N_{S^{\star}}(w_3)=\varnothing,$$
which implies that $$N_{S^{\star}}(w_1)\cup N_{S^{\star}}(w_3)=S^{\star}.$$
Similarly, we can prove that $$N_{T^{\star}}(w_2)\cap N_{T^{\star}}(w_3)=\varnothing~~\mbox{and}~~N_{T^{\star}}(w_2)\cup N_{T^{\star}}(w_3)=T^{\star}.$$
Let $S_1=N_{S^{\star}}(w_3)$ and $T_1=N_{T^{\star}}(w_3)$. Then $N_{S^{\star}}(w_1)=S^{\star}\backslash S_1$ and $N_{T^{\star}}(w_2)=T^{\star}\backslash T_1$ (see Fig. \ref{f5}). Next we shall prove that $|S_1||T_1|\leq1$. Suppose to the contrary that $|S_1||T_1|\geq2$. Since $G^{\star}\cong T_{n-3,2}$, there exists a path $P_3\subseteq G^{\star}[S_1\cup T_1]$. Then $G[V(P_3)\cup\{w_3\}]\cong B_2$, which contradicts that $G$ is $B_2$-free. Hence $G\in\mathscr{G}_1(B_2)$, which implies that
$\mathrm{Ex}_{3}(n,B_2)\subseteq\mathscr{G}_1(B_2)$.

\begin{figure}\label{f5}
\centering
\includegraphics[width=0.65\textwidth]{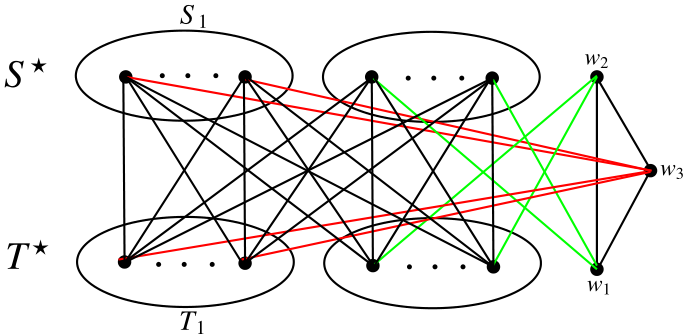}
\caption{Structure of the extremal graph $G$.}\label{f5}
\end{figure}

On the other hand, we shall prove that $\mathscr{G}_1(B_2)\subseteq \mathrm{Ex}_{3}(n,B_2)$. Let $G$ be an arbitrary graph in $\mathscr{G}_1(B_2)$. By the definition of $\mathscr{G}_1(B_2)$, $G$ is non-bipartite and $e(G)=\big\lfloor\frac{(n-1)^{2}}{4}\big\rfloor+2$. In the following, it suffices to prove that $G$ is $B_2$-free. Suppose to the contrary that $G$ contains a copy of $B_2$, say $H$. Since $G\backslash\{w_3\}$ is bipartite, $w_3\in V(H)$. If $d_{H}(w_3)=2$, then $H\backslash\{w_3\}\cong C_3$ is a subgraph of $G\backslash\{w_3\}$, a contradiction. Hence $d_{H}(w_3)=3.$ Then we have $H[N_{H}(w_3)]\cong P_3$, and hence $P_3\subseteq G\backslash\{w_3\}$. Recall that $G\backslash\{w_3\}$ is bipartite. Then $V(P_3)\cap S_1\neq\varnothing$ and $V(P_3)\cap T_1\neq\varnothing$, which implies that $|S_1||T_1|\geq2$. This contradicts that $|S_1||T_1|\leq1$. Hence $G$ is a non-bipartite $B_2$-free graph with $\big\lfloor\frac{(n-1)^{2}}{4}\big\rfloor+2$ edges, that is, $G\in \mathrm{Ex}_{3}(n, B_2)$. Therefore, $\mathscr{G}_1(B_2)\subseteq \mathrm{Ex}_{3}(n,B_2)$.

Combining the above two aspects, we have $\mathrm{Ex}_{3}(n,B_2)=\mathscr{G}_1(B_2)$ and $\mathrm{ex}_3(n, B_{2})=\big\lfloor\frac{(n-1)^{2}}{4}\big\rfloor+2$.

\vspace{2mm}
\noindent{\bf Case 2.} $r\geq2$.
\vspace{2mm}

According to (\ref{eq13}), we can obtain that
\begin{eqnarray}\label{eq14}
|N_{S^{\star}}(w_1)\cap N_{S^{\star}}(w_3)|\geq d_{S^{\star}}(w_1)+d_{S^{\star}}(w_3)-|S^{\star}|=r-1
\end{eqnarray}
and
\begin{eqnarray}\label{eq15}
|N_{T^{\star}}(w_2)\cap N_{T^{\star}}(w_3)|\geq d_{T^{\star}}(w_2)+d_{T^{\star}}(w_3)-|T^{\star}|=r-1.
\end{eqnarray}
Note that $G$ is $B_{r+1}$-free, $w_1w_3\in E(G)$ and $w_2\in N_{G}(w_1)\cap N_{G}(w_3)$. Then we have
\begin{eqnarray}\label{eq16}
|N_{S^{\star}}(w_1)\cap N_{S^{\star}}(w_3)|\leq r-1.
\end{eqnarray}
Similarly, we can obtain that
\begin{eqnarray}\label{eq17}
|N_{T^{\star}}(w_2)\cap N_{T^{\star}}(w_3)|\leq r-1.
\end{eqnarray}
It follows from (\ref{eq14})-(\ref{eq17}) that
\begin{eqnarray}\label{eq18}
|N_{S^{\star}}(w_1)\cap N_{S^{\star}}(w_3)|=r-1~~\mbox{and}~~|N_{T^{\star}}(w_2)\cap N_{T^{\star}}(w_3)|=r-1.
\end{eqnarray}

\begin{figure}\label{f6}
\centering
\includegraphics[width=0.65\textwidth]{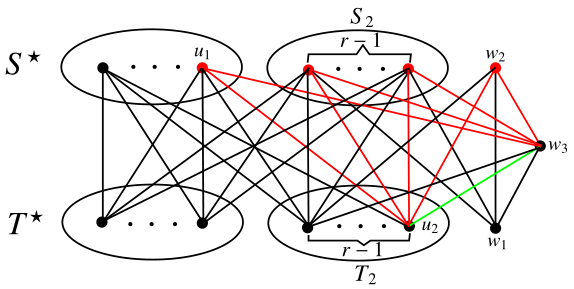}
\caption{A book $B_{r+1}$ in $G$.}\label{f6}
\end{figure}
Let $S_2=N_{S^{\star}}(w_1)\cap N_{S^{\star}}(w_3)$ and $T_2=N_{T^{\star}}(w_2)\cap N_{T^{\star}}(w_3)$. In the following, we further prove that $$N_{S^{\star}\backslash S_2}(w_3)=\varnothing.$$
Suppose to the contrary that there exists a vertex $u_1\in S^{\star}\backslash S_2$ such that $u_1$ is adjacent to $w_3$ (see Fig. \ref{f6}). Take a vertex $u_2\in T_2$. Then $u_2w_2$, $u_2w_3\in E(G).$ Recall that $G^{\star}\cong T_{n-3,2}$. Then every vertex in $S_2\cup\{u_1\}$ is adjacent to $u_2$. Hence $|N_{S^{\star}}(u_2)\cap N_{S^{\star}}(w_3)|\geq|S_2\cup\{u_1\}|=r$. Note that $w_2\in N_{G}(u_2)\cap N_{G}(w_3)$. Then we have $|N_{G}(u_2)\cap N_{G}(w_3)|\geq r+1$, which means that $G$ contains a $B_{r+1}$, a contradiction. Therefore, $N_{S^{\star}\backslash S_2}(w_3)=\varnothing$.
Combining (\ref{eq18}), we know that $d_{S^{\star}}(w_3)=r-1$. Furthermore, by (\ref{eq13}), we have $d_{S^{\star}}(w_1)=|S^{\star}|$.
Therefore, we can obtain that $$N_{S^{\star}}(w_1)=S^{\star}~~\mbox{and}~~N_{S^{\star}}(w_3)=S_{2}.$$
Similarly, we can prove that $N_{T^{\star}\backslash T_2}(w_3)=\varnothing.$ Furthermore, it follows that
$$N_{T^{\star}}(w_2)=T^{\star}~~\mbox{and}~~N_{T^{\star}}(w_3)=T_{2}.$$
This implies that $G[E(S^{\star}\cup\{w_2\}, T^{\star}\cup\{w_1\})]$ is isomorphic to $K_{\lfloor\frac{n-1}{2}\rfloor,\lceil\frac{n-1}{2}\rceil}$ and $|N_{S^{\star}\cup\{w_2\}}(w_3)|=|N_{T^{\star}\cup\{w_1\}}(w_3)|=r$. Hence $G\cong K_{\lfloor\frac{n-1}{2}\rfloor,\lceil\frac{n-1}{2}\rceil}^{r, r}$. Then we have $\mathrm{Ex}_{3}(n, B_{r+1})=\big\{K_{\lfloor\frac{n-1}{2}\rfloor,\lceil\frac{n-1}{2}\rceil}^{r, r}\big\}$ and $\mathrm{ex}_3(n, B_{r+1})=\big\lfloor\frac{(n-1)^{2}}{4}\big\rfloor+2r$ for $r\geq2$.
\hspace*{\fill}$\Box$

\section*{Declaration of competing interest}

The authors declare that they have no conflict of interest.

\section*{Data availability}

No data was used for the research described in the article.

\section*{Acknowledgements}

The research of Lu Miao is supported by CSC Program (No. 202407040080). The research of Ruifang Liu is supported by the National Natural Science Foundation of China (Nos. 12371361 and 12171440), Distinguished Youth Foundation of Henan Province (No. 242300421045).

\end{document}